\documentclass[12pt,oneside,reqno]{amsart}

\usepackage{latexsym,amssymb,enumerate, tikz,color,mathenv}
\input{xypic}
 \usepackage{fullpage}
 \usepackage{amsmath}

\usepackage[all]{xy}

\usepackage{color}

\usepackage[mathscr]{euscript}

\newcommand{\Hq}{\mathbb{H}}

\newcommand{\C}{\mathbb{C}}

\newcommand{\R}{\mathbb{R}}
\newcommand{\Z}{\mathbb{Z}}
\newcommand{\N}{\mathbb{N}}
\newcommand{\Q}{\mathbb{Q}}
\newcommand{\CP}{\mathbb{CP}}
\newcommand{\RP}{\mathbb{RP}}
\newcommand{\HP}{\mathbb{HP}}
\newcommand{\FP}{\mathbb{FP}}

\newcommand{\F}{\mathbb{F}}

\newcommand{\T}{\mathscr{T}}

\newcommand{\la}{\langle}
\newcommand{\ra}{\rangle}

\DeclareMathOperator{\Ca}{Ca\mathbb{P}^2}
\DeclareMathOperator{\Vect}{Vect}
\DeclareMathOperator{\Rep}{Rep}
\DeclareMathOperator{\Rr}{\mathfrak{R}}
\DeclareMathOperator{\U}{U}
\DeclareMathOperator{\Or}{O}
\DeclareMathOperator{\So}{SO}
\DeclareMathOperator{\Sp}{Sp}
\DeclareMathOperator{\Gr}{G}
\DeclareMathOperator{\Bo}{BO}

\DeclareMathOperator{\Spin}{Spin}

\newtheorem{thm}{Theorem}[section]
\newtheorem*{thm*}{Theorem}

\newtheorem{lem}[thm]{Lemma}

\newtheorem{prop}[thm]{Proposition}
\newtheorem{cor}[thm]{Corollary}

\newtheorem*{mainthm}{Main Theorem}

\theoremstyle{definition}

\newenvironment{question}{\noindent\textbf{Question:}}

\begin{document}

\author{David Gonz\'alez-\'Alvaro}
\address{ Department of Mathematics, Universidad Aut\'onoma de Madrid, and ICMAT CSIC-UAM-UCM-UC3M}
\curraddr{}
\email{dav.gonzalez@uam.es}


\thanks{The author was supported by the following research grants:  MTM2011-22612 from the Ministerio de Ciencia e Innovaci\'on (MCINN), FPI grant BES-2012-053704, MINECO: ICMAT Severo Ochoa project SEV-2011-0087, and MTM2014-57769-C3-3-P from the Ministerio de Econom\'ia y Competitividad of Spain.}

\thanks{}

\title[Nonnegative curvature on stable bundles]{Nonnegative curvature on stable bundles over compact rank one symmetric spaces}


\begin{abstract} 
In this note we show that every (real or complex) vector bundle over a compact rank one symmetric space carries, after taking the
Whitney sum with a trivial bundle of sufficiently large rank, 
a metric with nonnegative sectional curvature. We also examine the case of complex vector bundles over other manifolds, and give upper bounds for the rank of the trivial bundle that is necessary to add when the base is a sphere.



\end{abstract}

\maketitle


\section{Introduction and statement of results}

In 1972, Cheeger and Gromoll proved the fundamental structure theorem for open nonnegatively curved Riemannian manifolds:

\begin{thm*}[The Soul Theorem \cite{CG}]
Let $M$ be an open Riemannian manifold with nonnegative sectional curvature. There exists a compact, totally geodesic and totally convex submanifold $S$ without boundary such that $M$ is diffeomorphic to the normal bundle of $S$.
\end{thm*}

Such a submanifold is called a \textit{soul} of $M$. It is natural to ask to what extent a converse to the Soul Theorem holds.

\medskip

\begin{question}\label{existence}
Let $E$ be a vector bundle over a compact manifold $S$ with nonnegative sectional curvature. Does $E$ admit a complete metric of nonnegative curvature with soul $S$?
\end{question}

\medskip

The answer is clearly affirmative when $S$ is a homogeneous manifold $G/H$ of a compact Lie group $G$ and $E$ is a homogeneous vector bundle; that is, a bundle of the form $(G\times\F^n)/H$, where $\F$ stands for $\R$ or $\C$ and $H$ acts on $\F^n$ by means of a linear representation.

The first obstructions to the above question were found by \"Ozaydin and Walschap in \cite{OW}: a plane bundle over a torus admits a nonnegatively curved metric if and only if its rational Euler class vanishes.  Later, Guijarro in his thesis \cite{Gu} and Belegradek and Kapovitch in the series of papers \cite{BK1} and \cite{BK2}, extended these results to a larger class of bundles over some other nonsimply connected souls. 

However, in all these examples the obstructions are always due to the existence of a nontrivial fundamental group. So it is still important to see whether nonnegatively curved metrics exist when the base of the bundle is simply connected. 

Even the case of the sphere is still open, except for dimensions $n\leq 5$; see the article \cite{GZ1} by Grove and Ziller. So it is rather welcome to see that for any sphere there is a positive answer after passing to the stable realm.

\begin{thm*}[Rigas \cite{Ri}]
Let $E$ be a real vector bundle over a sphere $\mathbb{S}^n$. Denote by $k$ the trivial real vector bundle of rank $k$.  Then, for some $k$ the Whitney sum $E\oplus k$ admits a metric with nonnegative sectional curvature.
\end{thm*}

The starting point in Rigas' proof is the isomorphism between stable classes of real vector bundles over $\mathbb{S}^n$ and the homotopy group $\pi_n (\Bo)$, where $\Bo$ is the classifying space of the infinite orthogonal group $\Or$. He shows that the generators of $\pi_n (\Bo)$ can be realized by isometric embeddings of standard Euclidean spheres as totally geodesic submanifolds of Grassmannian manifolds. Using this fact he is able to prove the existence of homogeneous bundles in every stable class. Recall that two vector bundles $E,F$ over a compact space are stably equivalent if there exist trivial bundles $n,m$ such that $E\oplus n$ is isomorphic to $F\oplus m$.

Our goal is to extend Rigas' Theorem to some other nonnegatively curved compact spaces.  Natural candidates are the remaining compact rank one symmetric spaces, namely the projective spaces  $\RP^{n}$, $\CP^{n}$, $\HP^{n}$ and the Cayley plane $\Ca$. In order to do that, the main tool will be the isomorphism between stable classes and reduced $K$-theory.


$K$-theory of complex vector bundles over a topological space $X$ was introduced around 1960 by Atiyah and Hirzebruch (see \cite{At}); in \cite{AH} they studied more closely the particular case when $X$ is a compact homogeneous space. $K$-theory concerning real vector bundles has been also studied (see for example \cite{Ho}, \cite{San}), although it is not so well understood as in the complex case. The following is the main result in our paper.

\begin{mainthm}\label{principalthm cross}
Let $E$ be an arbitrary real (resp. complex) vector bundle over a compact rank one symmetric space. Denote by $k$ the trivial real (resp. complex) vector bundle of rank $k$. Then, for some $k$ the Whitney sum $E\oplus k$ admits a metric with nonnegative sectional curvature.
\end{mainthm}

In the case of the sphere our methods yield an alternative proof of Rigas' Theorem. Moreover, our approach allows us to give an upper bound $k_0$ for the smallest integer $k$ satisfying the Main Theorem. In order to state our result we need to recall that, as a consequence of the Bott Integrability Theorem (see \cite{Hu}, Chapter 20), if $E$ is a real vector bundle over a sphere $\mathbb{S}^n$ of dimension $n\equiv  0\ (\textrm{mod}\ 4)$, then its $(n/4)$-th Pontryagin class $p_{n/4}(E)$ is of the form
$$p_{n/4}(E)= ((n/2) - 1)!(\pm l_E) a$$
for some natural number $l_E$, where $a$ is a generator of $H^{n}(\mathbb{S}^{n},\Z)$.


\begin{thm}\label{thm sphere rank bounds}
Let $E$ be an arbitrary real vector bundle over $\mathbb{S}^n$. Let $k_0$ be the least integer such that the Whitney sum $E\oplus k_0$ admits a metric with nonnegative sectional curvature. The following inequalities hold:
\begin{itemize}

\item $k_0\leq n+1$, if $n\equiv 3,5,6,7\ (\textrm{mod}\ 8)$.

\item $k_0 \leq 2^n$, if $n\equiv  1,2\ (\textrm{mod}\ 8)$.

\item $k_0\leq \max\{n+1, 2^{n-1}l_E\}$, if $n\equiv  0, 4\ (\textrm{mod}\ 8)$.
\end{itemize}

\end{thm}

%
%
%
%
%
%

The results by Atiyah and Hirzebruch on $K$-theory of complex vector bundles over homogeneous spaces were extended by several authors (see for example \cite{Ad}, \cite{Ha}, \cite{Mi1}, \cite{Mi2}, \cite{Pi}). The following theorem will be a consequence of some of these results. 

\begin{thm}\label{thm complex bundles over many spaces}
Let $E$ be an arbitrary complex vector bundle over a manifold in one of the two following classes $X_i$:
\begin{itemize}
\item $X_1$ is the class of compact nonnegatively curved manifolds $M$ whose even dimensional Betti numbers $b_{2i}(M)$ vanish for $i\geq 1$, and such that $H^*(M,\Z)$ is torsion-free.
\item $X_2$ is the class of compact homogeneous spaces $G/H$ such that $G$ is a compact, connected Lie group with $\pi_1(G)$ torsion-free and $H$ a closed, connected subgroup of maximal rank.
\end{itemize}
Denote by $k$ the trivial complex vector bundle of rank $k$. Then, for some $k$ the Whitney sum $E\oplus k$ admits a metric with nonnegative sectional curvature.
\end{thm}


Odd-dimensional spheres belong to class $X_1$. The class $X_2$ includes such manifolds as even-dimensional spheres, complex and quaternionic Grassmannian manifolds, the Wallach flag manifolds $W^6$, $W^{12}$ and $W^{24}$ or the Cayley plane. Note that manifolds in the class $X_2$ inherit a nonnegatively curved metric from a biinvariant metric on $G$.

 
The paper is organized as follows. Section \ref{section stable classes repres} recalls basic definitions and facts about $K$-theory and stable classes of vector bundles, and relates them in the homogeneous setting. Section \ref{section proof thm complex bundles} contains the proof of Theorem \ref{thm complex bundles over many spaces}. Section \ref{section sphere} contains the proofs of the Main Theorem for the spheres and of Theorem \ref{thm sphere rank bounds}. The proofs of the Main Theorem for  projective spaces and the Cayley plane  are given in sections \ref{section grassmannian} and \ref{section cayley plane} respectively.

\medskip

\noindent\textbf{Acknowledgements:} This paper is part of the author's PhD. Thesis at the ICMAT and the Universidad Aut\'onoma de Madrid; he would like to thank his thesis advisor Luis Guijarro for letting him know about the problem. The author would also like to thank Fernando Galaz-Garc\'ia for useful comments and Gabino Gonz\'alez-Diez for helpful conversations.


\section{Stable classes and homogeneous bundles}\label{section stable classes repres}

Throughout this section $\F$ will denote either one of the fields $\R$ or $\C$. 

\subsection{Stable classes of vector bundles and $K_{\F}$-theory}\label{subsection stable classes}

We will denote by $\Vect_{\F}(M)$ the set of isomorphism classes of $\F$-vector bundles over a manifold $M$. The Whitney sum $\oplus$ and the tensor product of bundles $\otimes_{\F}$ endow $\Vect_{\F}(M)$ with a semiring strucure. Let 
 $$c: \Vect_{\R}(M)\rightarrow \Vect_{\C}(M)\quad \text{and} \quad r: \Vect_{\C}(M)\rightarrow \Vect_{\R}(M)$$ 
 be the complexification and the real restriction maps of vector bundles respectively. We will write $m_{\F}$ or just $m$ (when there is no danger of confusion) for the trivial $\F$-vector bundle of rank $m$; and $mE$ for the Whitney sum of $E$ with itself $m$ times.
 
If the manifold $M$ is compact we have the following well-known result (see e.g. Lemma 9.3.5 in \cite{AGP}).
 
\begin{lem}\label{lemma exist orthog bundle}
Let $E\in \Vect_{\F}(M)$ with $M$ compact. Then there exists $F\in\Vect_{\F}(M)$ such that $E\oplus F$ is isomorphic to a trivial bundle.
\end{lem}

From now on we assume that $M$ is compact. We say that $E,F\in\Vect_{\F}(M)$ are \emph{stably equivalent} if there exist trivial bundles $m_1,m_2$ such that $E\oplus m_1$ is isomorphic to $F\oplus m_2$. We will denote by $S_{\F}(M)$ the set of stable classes of bundles over $M$ and by $\{ E\}_{\F}$ the stable class of $E$. The Whitney sum gives $S_{\F}(M)$ the structure of an abelian semigroup. Furhtermore, by Lemma \ref{lemma exist orthog bundle}, every element $\{ E\}_{\F}$ has an inverse, so $S_{\F}(M)$ is an abelian group. Later on we will use the following theorem (see e.g. \cite{Hu}, Chapter 9).

\begin{thm}\label{thm stability of bundles high rank}
Let $E$ and $F$ be real vector bundles of the same rank $k$ over a compact $n$-dimensional manifold $M$ such that $E\oplus m$ is isomorphic to $F\oplus m$ for some integer $m$. If $k\geq n+1$, then $E$ and $F$ are isomorphic.
\end{thm}

We write $K_{\F}(M)$ for the $K$-theory ring of $\F$-vector bundles over $M$. This is the ring completion of the semiring $\Vect_{\F}(M)$. Its elements, called \textit{virtual bundles}, are usually written in the form $[E]-[F]$, where $[E_1]-[F_1]$ equals $[E_2]-[F_2]$ if there exists another bundle $E_3$ such that $E_1\oplus F_2\oplus E_3$ and $E_2\oplus F_1\oplus E_3$ are isomorphic. Observe that $K_{\F}(M)$ is a commutative ring with unity.

When $M$ is compact, every element in $K_{\F}(M)$ can be written in the form $[E]-[m]$. To prove this, choose a virtual bundle $[E_1]-[F_1]$. By Lemma \ref{lemma exist orthog bundle} there exists a vector bundle $F_{1}^{\perp}$ such that $F_{1}\oplus F_{1}^{\perp} = m$. Then clearly $[E_1]-[F_1]$ equals $[E_1 \oplus F_{1}^{\perp}]-[m]$.

Consider the ring homomorphism $d:K_{\F}(M)\rightarrow\Z$ given by $d([E]-[F])=\text{rank}(E)-\text{rank}(F)$. The kernel of $d$  is called the \emph{reduced} $K$-theory ring and it is usually denoted by $\tilde{K}_{\F}(M)$. It is an ideal of $K_{\F}(M)$ and thus a ring without unity. There is a natural splitting $K_{\F}(M)=\tilde{K}_{\F}(M)\oplus\Z$. We recall the following well-known theorem that relates the two latter constructions (see e.g. Theorem 9.3.8 in \cite{AGP}).

\begin{thm}\label{theorem isom stable ktheory}
Let $M$ be a compact manifold. Then $\tilde{K}_{\F}(M)\approx S_{\F}(M)$ as abelian groups. An isomorphism is given by:
$$
\begin{array}{rcc}
\Phi_{\F}:\tilde{K}_{\F}(M)  & \rightarrow & S_{\F}(M) \\
\left[ E\right]-\left[ m\right] & \mapsto &  \{ E\}_{\F}
\end{array}
$$
\end{thm}

To simplify notation, from now on $E-F$ will denote the virtual bundle $[E]-[F]$. More details about these concepts can be found in \cite{AGP}, \cite{At} and \cite{Hu}. 


\subsection{Homogeneous vector bundles}

Let $G$ be a Lie group. Denote by $\Rep_{\F}(G)$ the set of isomorphism classes of $\F$-representations of $G$. The direct sum and the tensor product of representations endow $\Rep_{\F}(G)$ with a semiring strucure. Let 
 $$c: \Rep_{\R}(G)\rightarrow \Rep_{\C}(G)\quad \text{and} \quad r: \Rep_{\C}(G)\rightarrow \Rep_{\R}(G)$$ 
stand for complexification and real restriction of representations. We will write $m_{\F}$ or simply $m$ for the trivial representation of $G$ on $\F^m$; and $m\rho$ for the sum of $\rho\in\Rep_{\F}(G)$ with itself $m$ times.

If $i:H\rightarrow G$ is the inclusion of a closed subgroup $H$, we denote by 
$$i_{\F}^*: \Rep_{\F}(G)\rightarrow \Rep_{\F}(H)$$
the semiring homomorphism defined by restricting representations of $G$ to $H$.

For each $\rho \in \Rep_{\F}(H)$ we have the diagonal action of $H$ on $G\times \F^m$ from the right given by
$$\begin{array}{ccc}
 (G\times \F^m)\times H & \longrightarrow & G\times \F^m \\
  ((g,v),h) & \longmapsto & (gh,\rho(h)^{-1}v) \end{array}$$
where $m$ is the dimension of the representation $\rho$. The quotient space  $E_{\rho}:=(G\times \F^m)/H$ is the total space of an associated vector bundle $\pi_{\rho}:E_{\rho}\rightarrow G/H$ over the homogeneous manifold $G/H$, where $\pi_{\rho}$ is the obvious projection map. Vector bundles arising in this way are called \emph{homogeneous}. 

We have an analogue result to Lemma \ref{lemma exist orthog bundle} in the homogeneous setting. More precisely (see \cite{Se}), we have the following:

\begin{lem}\label{lemma exist orthog bundle homogeneous}
Let $G$ be a compact Lie group, $H$ a closed subgroup and $E_\rho\in \Vect_{\F}(G/H)$ a homogeneous bundle. Then there exists a representation $\rho^{\perp}\in \Rep_{\F}(H)$ such that $E_{\rho}\oplus E_{\rho^{\perp}}$ is isomorphic to a trivial bundle.
\end{lem}

Recall that $E_{\rho}$ is isomorphic to a trivial bundle if and only if $\rho$ is the restriction to $H$ of a representation of $G$ (see \cite{GOV}, page 131), i.e., if $\rho=i_{\F}^{*}(\tau)$ for some $\tau\in\Rep_{\F}(G)$.

It is straightforward to check that the following map is a morphism of semirings
$$\begin{array}{ccc}
{\alpha_{\F}}:\Rep_{\F}(H) & \rightarrow & \Vect_{\F}(G/H) \\
 \ \ \ \ \rho  & \mapsto &  E_\rho   \end{array}$$
 
Composing $\alpha_{\F}$ with the map $\{ \}_{\F}:\Vect_{\F}(G/H)\rightarrow S_{\F}(G/H)$ that assigns stable classes to vector bundles we get the induced morphism of semigroups
  $$\begin{array}{ccc}
\{\alpha\}_{\F}: \Rep_{\F}(H) & \rightarrow & S_{\F}(G/H) \\
  \ \ \ \  \rho  & \mapsto &   \left\{ E_\rho \right\}_{\F}  \end{array}$$
  
The ring completion $\Rr_{\F}(G)$ of the semiring $\Rep_{\F}(G)$ is defined in the same manner as the ring completion $K_{\F}(M)$ of the semiring $\Vect_{\F}(M)$. The semiring morphisms $r,c,i_{\F}^{*}$ and ${\alpha_{\F}}$ extend to ring morphisms of the corresponding ring completions, which we denote in the same way. We will write $\rho_1 \rho_2$ and $\rho_1 +\rho_2$ (resp. $E_1 E_2$ and  $E_1 + E_2$) to denote the multiplication and the sum laws in $\Rr_{\F}(G)$ (resp. $K_{\F}(M)$) induced from tensor product and direct sum of representations (resp. vector bundles). The following diagrams commute:

\[
\xymatrix{ &  \Rr_{\R}(G)\ar[d]^{c} \ar[r]^{i_{\R}^*} & \Rr_{\R}(H)\ar[d]^{c} & & \Rr_{\R}(H)\ar[d]^{c} \ar[d]^{c}\ar[r]^{\alpha_{\R}} & K_{\R}(G/H)\ar[d]^{c}\\
 & \Rr_{\C}(G) \ar[r]^{i_{\C}^*} & \Rr_{\C}(H)& & \Rr_{\C}(H)\ar[r]^{\alpha_{\C}}  & K_{\C}(G/H)} 
\]
 
\medskip
 
The maps $\{\alpha\}_{\F}:\Rep_{\F}(H) \rightarrow S_{\F}(G/H)$ and ${\alpha}_{\F}:\Rr_{\F}(H)\rightarrow K_{\F}(G/H)$ are related, as shown in the lemma below. Denote by $\tilde{\Rr}_{\F}(H)$ the kernel of the map $d:\Rr_{\F}(H)\rightarrow \Z$ defined by $d(\rho_1 -\rho_2)=\dim{\rho_1}-\dim{\rho_2}$. It is an ideal of $\Rr_{\F}(H)$.

\begin{lem}\label{lem relates images beta} 
Let $G$ be a compact Lie group and $H$ a closed subgroup. Then, with the notations above,
\begin{enumerate}
\item $\alpha_{\F}( \tilde{\Rr}_{\F}(H))\subset \tilde{K}_{\F}(G/H)$, and if the map $\alpha_{\F}: \Rr_{\F}(H)\rightarrow K_{\F}(G/H)$ is surjective, then the restriction $\alpha_{\F}: \tilde{\Rr}_{\F}(H)\rightarrow \tilde{K}_{\F}(G/H)$ is also surjective.
\item The following equality holds:
$$\Phi_{\F}\circ\alpha_{\F}( \tilde{\Rr}_{\F}(H)) = \{\alpha\}_{\F}(\Rep_{\F}(H)),$$
where $\Phi_{\F}$ is the map from Theorem \ref{theorem isom stable ktheory}. In particular, if $\alpha_{\F}$ is surjective, then $\{\alpha\}_{\F}$ is also surjective.
\end{enumerate} 
\end{lem}

\begin{proof}
The first statement follows inmediatly from the definition of $\alpha_{\F}$. 

As for the second part, the inclusion $\{\alpha\}_{\F}(\Rep_{\F}(H))\subset\Phi_{\F}\circ\alpha_{\F}( \tilde{\Rr}_{\F}(H))$ is obvious.

Now, every element in $\alpha_{\F}( \tilde{\Rr}_{\F}(H))$ is of the form $E_{\rho_1}-E_{\rho_2}$, for some $\rho_1 ,\rho_2\in\Rep_{\F}(H)$ satisfying $\dim{\rho_1}=\dim{\rho_2}$. By Lemma \ref{lemma exist orthog bundle homogeneous} there exists $\rho_2^\perp\in\Rep_{\F}(H)$ such that $E_{\rho_2}\oplus E_{\rho_2^{\perp}}=m$. Thus
$$E_{\rho_1}-E_{\rho_2}= E_{\rho_1}-E_{\rho_2}+E_{\rho_2^{\perp}}-E_{\rho_2^{\perp}}=E_{\rho_1\oplus\rho_2^{\perp}}-m,$$ 
hence we have 
$$\Phi_{\F}\left( E_{\rho_1}-E_{\rho_2} \right) =\Phi_{\F}\left(E_{\rho_1\oplus\rho_2^{\perp}}-m\right) = \left\{ E_{\rho_1\oplus\rho_2^{\perp}}\right\}_{\F}$$
\end{proof}

We recall the following theorem by Pittie which relates the complex representation and $K$-theory rings of a certain class of homogeneous spaces.

\begin{thm}[Pittie, \cite{Pi}]\label{theorem free G surjection}
Let $G$ be a compact, connected Lie group such that $\pi_1(G)$ is torsion free. Let $H$ be a closed, connected subgroup of maximal rank. Then the homomorphism 
$$\alpha_{\C}: \Rr_{\C}(H)\rightarrow K_{\C}(G/H)$$
is surjective.
\end{thm}


\subsection{Nonnegative sectional curvature} Let $G/H$ be a homogeneous manifold. If $H$ is compact, then for every $\rho\in \Rep_{\F}(H)$ we can assume that $r(\rho(H))$ lies in some orthogonal group $\Or(n)$. Suppose that $G$ admits a metric $\la \, ,\,\ra_{G}$ of nonnegative sectional curvature which is invariant under the action of $H$ from the right (for instance a biinvariant metric in the case of compact $G$), hence inducing a nonnegatively metric on the quotient manifold $G/H$ by O'Neill's Theorem on Riemannian submersions. Endow $G\times \F^n$ with the product metric of $\la \,,\,\ra_{G}$ and the flat Euclidean metric. Then, again by O'Neill's Theorem on Riemannian submersions, $E_{\rho}$ inherits a quotient metric of nonnegative curvature of which $G\times_H \{ 0\}=G/H$ is a soul.

Now suppose that there is a homogeneous bundle in every stable class $S_{\F}(G/H)$. Then, for an arbitrary  $\F$-vector bundle  $E$ over $G/H$ there exist $\rho\in \Rep_{\F}(H)$ and $n,m\in\N$ such that 
$$E\oplus n = E_{\rho}\oplus m =  E_{\rho\oplus m }$$
Therefore $E\oplus n$ is a homogeneous vector bundle and it admits a metric with nonnegative sectional curvature. We have proved:

\begin{lem}\label{lemma nonnegative curvature}
Let $G$ be a compact Lie group and $H$ a closed subgroup. Suppose that there is a homogeneous vector bundle in every stable class $S_{\F}(G/H)$. Then for every $\F$-vector bundle $E$ there exists $k\in\N$ such that $E\oplus k_{\F}$ admits a metric with nonnegative sectional curvature.
\end{lem}


\section{Proof of Theorem \ref{thm complex bundles over many spaces}}\label{section proof thm complex bundles}

The Chern character induces a ring homomorphism $ch: K_{\C}(M)\rightarrow H^{*}(M,\Q)$. Atiyah and Hirzebruch studied extensively this homomorphism in \cite{AH}. A consequence of their results is the following

\begin{thm}[\cite{AH}]\label{thm atiyah hirze even dimen betti numbers}
Let $M$ be a compact manifold. Then $K_{\C}(M)$ is additively a finitely generated abelian group, and its rank equals the sum of the even-dimensional Betti numbers of $M$. Moreover, if $H^*(M,\Z)$ is torsion-free, then $K_{\C}(M)$ is free abelian, i.e.,
$$K_{\C}(M)=\underbrace{\Z\oplus \dots\oplus\Z}_{n\ \text{times}}, $$
where $n$ is the sum of the even-dimensional Betti numbers. 
\end{thm}


Theorem \ref{thm atiyah hirze even dimen betti numbers} implies that manifolds $M$ in the class $X_1$ satisfy that $K_{\C}(M)=\Z$, and therefore $\tilde{K}_{\C}(M)=0$. Thus every complex vector bundle $E$ is stably trivial, i.e., for some integer $k$ the Whitney sum $E\oplus k_{\C}$ is isomorphic to a trivial bundle $M\times\C^{k'}$, and hence the product metric has nonnegative sectional curvature.

Theorem \ref{theorem free G surjection} applies directly to manifolds in the class $X_2$, and then Lemma \ref{lem relates images beta} together with Lemma \ref{lemma nonnegative curvature} completes the proof.


\section{The spheres $\mathbb{S}^n$}\label{section sphere}

As a homogeneous space, the sphere can be viewed as 
$$\mathbb{S}^{n}=\So(n+1)/\So(n)=\Spin(n+1)/\Spin(n).$$
Recall that the spin group $\Spin(n)$ is the double cover of the special orthogonal group $\So(n)$.  For $n>2$, the group $\Spin(n)$ is simply connected and so coincides with the universal cover of $\So(n)$.

\subsection{Representation rings of $\Spin(n)$}\label{subsection representation rings of spin}
 Denote by $\Lambda^1$ the canonical representation of $\So(n)$ in $\R^n$ and by $\Lambda^{k}$ the $k$-th exterior product of $\Lambda$. As usual, we set $\Lambda^0 =1$. Abusing notation, denote also by $\Lambda^{k}$ its complexification $c(\Lambda^{k})$. These representations induce representations of $\Spin(n)$ via the double covering map $\Spin(n)\rightarrow \So(n)$ which are usually denoted in the same way.
 
The representation rings of $\Spin(n)$ are known (see \cite{AdMM} and \cite{BD}, chapter VI). In the odd case $\Rr_{\C}(\Spin(2n+1))$ equals the polynomial ring:
$$\Rr_{\C}(\Spin(2n+1))=\Z [\Lambda^1,\dots,\Lambda^{n-1},\Delta].$$
The special $2^n$-dimensional representation $\Delta$ satisfies:
$$\Delta\Delta=1+\Lambda^1+\dots+\Lambda^{n-1}+\Lambda^{n}$$
In the even case, $\Rr_{\C}(\Spin(2n))$ is also a polynomial ring, namely 
$$\Rr_{\C}(\Spin(2n))=\Z [\Lambda^1,\dots,\Lambda^{n-2},\Delta_{+},\Delta_{-}].$$
The special $2^{n-1}$-dimensional representations $\Delta_{+}$, $\Delta_{-}$ satisfy:
\begin{eqnarray}
\Delta_+\Delta_+ &= &\Lambda_+^n+\Lambda^{n-2}+\Lambda^{n-4}+\dots \\
\Delta_+\Delta_- &= &\Lambda^{n-1}+\Lambda^{n-3}+\Lambda^{n-5}+\dots \label{eqnarray repr plus tensor minus} \\ 
\Delta_-\Delta_- &= &\Lambda_-^n+\Lambda^{n-2}+\Lambda^{n-4}+\dots
\end{eqnarray}
where $\Lambda_{+}^{n}$ and $\Lambda_{-}^{n}$ are irreducible representations such that $\Lambda_{+}^{n}+\Lambda_{-}^{n}=\Lambda^n$. The sums end in $\Lambda^2+1$ or $\Lambda^3+\Lambda^1$ depending on the parity of $n$.

The irreducible representations $\Lambda^{k}$ with $k\leq n-1$ (resp. $k\leq n-2$) of $\Spin(2n+1)$ (resp. $\Spin(2n)$) are real, meaning that they are the complexification of a real representation. Moreover (see \cite{BD}, chapter VI), we have the following:

\begin{prop}\label{propo type of representation spin}
For $n\equiv m\ (\textrm{mod}\ 8)$, the special representations $\Delta$, $\Delta_+$ and $\Delta_-$ of $\Spin(n)$ have the following type:

\begin{center}
\begin{tabular}{  c | c | c | c | c | c | c | c | c }
                  
  m & 0 & 1 & 2 & 3 & 4 & 5 & 6 & 7\\
   \hline    
 $\text{Type}$ & $\R$ & $\R$ & $\C$ & $\Hq$ & $\Hq$ & $\Hq$ & $\C$ & $\R$ \\

\end{tabular} 
\end{center}
\end{prop}

In the case when $\Delta_+$, $\Delta_-$ or $\Delta$ are of real type we denote both the underlying real representation (not to be mistaken for the real restriction) and its complexification in the same way. 
 

Consider the standard inclusion 
$$\begin{array}{ccc} 
\So(n) & \rightarrow & \So(n+1)\\
\  \ A & \mapsto & \left(\begin{array}{cc}
A & 0 \\
0 & 1
\end{array}\right)
\end{array}$$ 
and its covering group homomorphism $i_n:\Spin(n)\rightarrow\Spin(n+1)$. The following relations hold (see \cite{BD}, chapter VI):
\begin{eqnarray}
 && i_{2n,\C}^*(\Lambda^{k}) = \Lambda^k+\Lambda^{k-1}\quad \text{for}\  1\leq k\leq n\label{eqnarray repr exterior derivatives even} \\
& &i_{2n,\C}^*(\Delta)=\Delta_+ +\Delta_-\label{eqnarray repr sum of plus minus special} \\
& & i_{2n-1,\C}^*(\Lambda^{k}) = \Lambda^k+\Lambda^{k-1}\quad \text{for}\ 1\leq k\leq n-1\label{eqnarray repr exterior derivatives odd} \\
& & i_{2n-1,\C}^*(\Lambda_{\pm}^{n})=\Lambda^{n-1} \\
 & & i_{2n-1,\C}^*(\Delta_{\pm})=\Delta \label{eqnarray repr delta special} 
\end{eqnarray}
Thus we get identities on the corresponding stable classes of complex vector bundles over the sphere:
\begin{cor}\label{cor stable classes sphere spin}
The following relations hold:
\begin{itemize}
\item Over  $\mathbb{S}^{2n}=\Spin(2n+1)/\Spin(2n)$,
\begin{eqnarray}
 & & \{ E_{\Lambda^k}\}_{\C} = \{ 1\}_{\C}\quad \text{for}\ 1\leq k\leq n\label{eqnarray exterior derivatives even} \\
  & &\{ E_{\Delta_{+}}\}_{\C}+ \{ E_{\Delta_{-}}\}_{\C}= \{ 1\}_{\C}\label{eqnarray sum of plus minus special} \\
 & &\{ E_{\Delta_{+}\otimes_{\C}\Delta_{-}}\}_{\C} = \{ 1\}_{\C}\label{eqnarray plus tensor minus} \\
& & \{ E_{\Delta_{+}\otimes_{\C}\Delta_{+}}\}_{\C}=2^n \{ E_{\Delta_{+}}\}_{\C} \label{eqnarray plus tensor plus special} 
\end{eqnarray}
\item Over  $\mathbb{S}^{2n-1}=\Spin(2n)/\Spin(2n-1)$,
\begin{eqnarray}
 & & \{ E_{\Lambda^k}\}_{\C} = \{ 1\}_{\C}\quad \text{for}\ 1\leq k\leq n-1\label{eqnarray exterior derivatives odd} \\
  & &\{ E_{\Delta}\}_{\C} = \{ 1\}_{\C}\label{eqnarray delta special} 
\end{eqnarray}
\end{itemize}
\end{cor}

\begin{proof}
 The relations (\ref{eqnarray sum of plus minus special}) and (\ref{eqnarray delta special}) follow immediately from (\ref{eqnarray repr sum of plus minus special}) and (\ref{eqnarray repr delta special}). The relations (\ref{eqnarray exterior derivatives even}) and (\ref{eqnarray exterior derivatives odd}) follow recursively from (\ref{eqnarray repr exterior derivatives even}) and (\ref{eqnarray repr exterior derivatives odd}) respectively, since $\Lambda^0=1=i_{2n,\C}^*(1)$. The latter, together with  (\ref{eqnarray repr plus tensor minus}), gives us (\ref{eqnarray plus tensor minus}). Finally, observe that 
\begin{eqnarray}
\Delta_+\Delta_+ +\Delta_-\Delta_- &=& \Lambda_+^n+ \Lambda_-^n+2\Lambda^{n-2}+2\Lambda^{n-4}+\dots\nonumber \\
&=&\Lambda^n +2\Lambda^{n-2}+2\Lambda^{n-4}+\dots\nonumber  \\
&=& i_{2n,\C}^*(\rho) \label{eqnarray triviality of plusplus minusminus}
\end{eqnarray}
for some $\rho\in\Rep_{\C}(\Spin(2n+1))$. On the other hand, by  (\ref{eqnarray repr sum of plus minus special}) we have that $\Delta_- =i_{2n,\C}^*(\Delta)- \Delta_+$, hence:
\begin{eqnarray}\label{equation plusplus minusminus}
\Delta_+\Delta_+ +\Delta_-\Delta_- &=& \Delta_+\Delta_+ + (i_{2n,\C}^*(\Delta)- \Delta_+)(i_{2n,\C}^*(\Delta)- \Delta_+)\nonumber \\
&=& 2\Delta_+\Delta_+ +  i_{2n,\C}^*(\Delta\Delta) - 2i_{2n,\C}^*(\Delta)\Delta_{+}
\end{eqnarray}
Combining (\ref{eqnarray triviality of plusplus minusminus}) and (\ref{equation plusplus minusminus}) we get
$$2\Delta_+\Delta_+ + i_{2n,\C}^*(\Delta\Delta)= i_{2n,\C}^*(\rho) + 2i_{2n,\C}^*(\Delta)\Delta_{+}$$
which proves (\ref{eqnarray plus tensor plus special}) since $E_{i_{2n,\C}^*(\Delta)}=E_{\dim{\Delta}}=2^n$.
\end{proof}

\subsection{The $K$-theory of the sphere}

The rings $K_{\F}(\mathbb{S}^{n})$ are well known (see \cite{MT}, chapter IV). In the complex case:
$$\tilde{K}_{\C}(\mathbb{S}^{2n+1})=0, \qquad \tilde{K}_{\C}(\mathbb{S}^{2n})=\Z. $$
In the real case:
$$
\begin{array}{lcl}
\tilde{K}_{\R}(\mathbb{S}^{8n})=\Z, & \qquad & 
\tilde{K}_{\R}(\mathbb{S}^{8n+4})=\Z, \\
\tilde{K}_{\R}(\mathbb{S}^{8n+1})=\Z_2, & \qquad & \tilde{K}_{\R}(\mathbb{S}^{8n+5})=0, \\
\tilde{K}_{\R}(\mathbb{S}^{8n+2})=\Z_2, & \qquad & \tilde{K}_{\R}(\mathbb{S}^{8n+6})=0, \\
\tilde{K}_{\R}(\mathbb{S}^{8n+3})=0, & \qquad & 
\tilde{K}_{\R}(\mathbb{S}^{8n+7})=0. 
\end{array}
$$

\medskip
 
\subsection{Proof of the Main Theorem for $\mathbb{S}^n$}

\begin{prop}\label{propo stable sphere}
The map
$$\{\alpha\}_{\F}: \Rep_{\F}(\Spin(n)) \rightarrow  S_{\F}(\mathbb{S}^n) $$
is surjective for all $n\in\N$ both in the real and in the complex case. Moreover, the stable classes in the cases in which $\tilde{K}_{\F}(\mathbb{S}^ {n})\neq 0$ are given by 
$$
\begin{array}{lcl}
S_{\C}(\mathbb{S}^ {2n})=\Z\{ E_{\Delta_+}\}_{\C}, & \qquad &  \\
S_{\R}(\mathbb{S}^ {8n})=\Z\{ E_{\Delta_+}\}_{\R}, & \qquad & S_{\R}(\mathbb{S}^ {8n+2})=\{ \{ 1\}_{\R}, \{ E_{r(\Delta_+ )}\}_{\R} \}, \\
S_{\R}(\mathbb{S}^ {8n+1})=\left\{ \{ 1\}_{\R}, \{ E_{\Delta}\}_{\R} \right\}, &  \qquad & S_{\R}(\mathbb{S}^ {8n+4})=\Z\{ E_{r(\Delta_+ )}\}_{\R}.
\end{array}
$$

\end{prop}

\begin{proof} The surjectivity of $\{\alpha\}_{\C}$ is included in Theorem \ref{thm complex bundles over many spaces}. From Corollary \ref{cor stable classes sphere spin} it follows that $S_{\C}(\mathbb{S}^ {2n})=\Z\{ E_{\Delta_+}\}_{\C}$.

The surjectivity of $\{\alpha\}_{\R}$ when $n\equiv 3,5,6,7\ (\textrm{mod}\ 8)$ is trivial since $\tilde{K}_{\R}(\mathbb{S}^n)=0$. 

Now let $E$ be an arbitrary real vector bundle over $\mathbb{S}^{n}$ for the remaining cases:

\begin{itemize}
\item $n\equiv  0\ (\textrm{mod}\ 8)$. By Theorem 5.12 in \cite{MT}, chapter IV, the map
$$c: \tilde{K}_{\R}(\mathbb{S}^{n}) \rightarrow  \tilde{K}_{\C}(\mathbb{S}^{n})\cong S_{\C}(\mathbb{S}^ {n})=\Z\{ E_{\Delta_+}\}_{\C} $$
is an isomorphism. From Proposition \ref{propo type of representation spin} we know that $\Delta_{+}$ is real and therefore 
$$ S_{\R}(\mathbb{S}^ {n})=\Z\{ E_{\Delta_+}\}_{\R}.$$

\item $n\equiv 2,4\ (\textrm{mod}\ 8)$. By Theorem 6.1 in \cite{MT}, chapter IV, the real restriction map  for $n\equiv 2\ (\textrm{mod}\ 8)$ (resp. $n\equiv 4\ (\textrm{mod}\ 8)$)
$$r:\tilde{K}_{\C}(\mathbb{S}^{n})\rightarrow \tilde{K}_{\R}(\mathbb{S}^{n})$$
is surjective (resp. an isomorphism). Therefore 
$$ 
\begin{array}{l}
S_{\R}(\mathbb{S}^ {n})=\{ \{ 1\}_{\R}, \{ E_{r(\Delta_+ )}\}_{\R} \}\ \text{ if }  n\equiv 2\ (\textrm{mod}\ 8).\\
 S_{\R}(\mathbb{S}^ {n})=\Z\{ E_{r(\Delta_+ )}\}_{\R}\ \text{ if }  n\equiv 4\ (\textrm{mod}\ 8).
 \end{array}
 $$
 
\item $n\equiv 1\ (\textrm{mod}\ 8)$. By Proposition \ref{propo type of representation spin} the representation $\Delta$ is real. We are going to prove that $\{ E_{\Delta}\}_{\R}$ is not trivial and hence 
$$S_{\R}(\mathbb{S}^ {8n+1})=\Z_2 =\{ \{ 1\}_{\R}, \{ E_{\Delta}\}_{\R} \}.$$

Denote by $i_{\F}^{*}$ the map $i_{8n+1,\F}^{*}$. We want to see that there does not exist $\tau\in\Rep_{\R}(\Spin(8n+2))$ such that $i_{\R}^*(\tau)=\Delta + k$, for any natural number $k$. Suppose it does; then $c(\tau)\in\Rep_{\C}(\Spin(8n+2))$ is of the form
$$
c(\tau)  = \displaystyle\sum_{j_1,\dots,j_{4n+1}}a_{j_1,\dots,j_{4n+1}}(\Lambda^{1})^{j_1}\dots (\Lambda^{4n-1})^{j_{4n-1}}(\Delta_+)^{j_{4n}}(\Delta_-)^{j_{4n+1}} $$
We can rewrite this expression as
$$c(\tau)  = \displaystyle\sum_{l_1,l_2}b_{l_1,l_2}(\Lambda^{1},\dots ,\Lambda^{4n-1})(\Delta_+)^{l_{1}}(\Delta_-)^{l_{2}}
$$
for the obvious polynomials $b_{l_1 ,l_2}\in\Z[\Lambda^1,\dots,\Lambda^{4n-1} ]$. Now we have:
$$
\begin{array}{rcl}
i_{\C}^{*} \left( c(\tau )\right) & = & \displaystyle\sum_{l_1,l_2}i_{\C}^{*}\left(a_{l_1,l_2}(\Lambda^{1},\dots ,\Lambda^{4n-1})\right)i_{\C}^{*}(\Delta_+)^{l_{1}}i_{\C}^{*}(\Delta_-)^{l_{2}} \\
 & = & \displaystyle\sum_{l_1,l_2}a_{l_1,l_2}(\Lambda^{1}+1,\dots ,\Lambda^{4n-1}+\Lambda^{4n-2})(\Delta)^{l_{1}+l_{2}}
\end{array}
$$
On the other hand,
$$c\left(i_{\R}^*(\tau)\right)=c(\Delta + k)=\Delta + k\in\Rr_{\C}(\Spin(8n+1))$$
From the identity $i_{\C}^{*}\circ c = c \circ i_{\R}^{*}$, it follows that
$$
\left\{
\begin{array}{l}
a_{0,0}(\Lambda^{1}+1,\dots ,\Lambda^{4n-1}+\Lambda^{4n-2}) =k \\
a_{1,0}(\Lambda^{1}+1,\dots ,\Lambda^{4n-1}+\Lambda^{4n-2})+a_{0,1}(\Lambda^{1}+1,\dots ,\Lambda^{4n-1}+\Lambda^{4n-2}) =1 \\
a_{i,j}(\Lambda^{1}+1,\dots ,\Lambda^{4n-1}+\Lambda^{4n-2})=0\ \text{ if } i+j\geq 2
\end{array}
\right.
$$
The map $\phi:\Z[\Lambda^1,\dots,\Lambda^{4n-1} ]\rightarrow \Z[\Lambda^1,\dots,\Lambda^{4n-1} ]$ defined by the rule $\phi (\Lambda^{k})=\Lambda^{k}+\Lambda^{k-1}$ for $k\geq 1$, is a ring isomorphism. The inverse is given recursively as $\phi^{-1} (\Lambda^{k})=\Lambda^{k}-\phi^{-1}(\Lambda^{k-1})$, where $\phi^{-1}(\Lambda^{1})=\Lambda^1 -1$. Therefore we have that 
$$
\left\{
\begin{array}{l}
a_{0,0}(\Lambda^{1},\dots ,\Lambda^{4n-1}) =k \\
a_{1,0}(\Lambda^{1},\dots ,\Lambda^{4n-1})+a_{0,1}(\Lambda^{1},\dots ,\Lambda^{4n-1}) =1 \\
a_{i,j}(\Lambda^{1},\dots ,\Lambda^{4n-1})=0\ \text{ if } i+j\geq 2
\end{array}
\right.
$$

We deduce that $c(\tau)$ equals either $k +\Delta_{+}$ or $k +\Delta_{-}$. It then would follow that either $\Delta_{+}$ or $\Delta_{-}$ is in the image of the complexification map. But this is a contradiction since as we can see in Proposition \ref{propo type of representation spin}, the representations $\Delta_{+}$ and $\Delta_{-}$ are not of real type.

Finally, let $d$ be the dimension of the real representation $\Delta$. Observe that
$$2\left( E_{\Delta} - d_\R \right)=r\circ c\left( E_{\Delta} - d_\R \right)=r\left( E_{\Delta} - d_\C \right)=0,$$ 
since $r: \tilde{K}_{\C} (\mathbb{S}^ {8n+1})\rightarrow  \tilde{K}_{\R} (\mathbb{S}^ {8n+1})$ is the zero map. It follows that $2\{ E_\Delta\}_{\R}=\{ 1\}_{\R}$.
\end{itemize}

\end{proof}

\subsection{Proof of Theorem \ref{thm sphere rank bounds}}

The proof follows from Proposition \ref{propo stable sphere} together with Theorem \ref{thm stability of bundles high rank}. Let $E$ be an arbitrary real vector bundle over the sphere $\mathbb{S}^n$. If $E$ is stably trivial, then the Whitney sum $E\oplus k$ is isomorphic to a trivial bundle if $\text{rank}(E\oplus k) \geq n+1$.

\begin{itemize}

\item  $n\equiv 3,5,6,7\ (\textrm{mod}\ 8)$. Since $\tilde{K}_{\R}(\mathbb{S}^n)=0$, every bundle is stably trivial so $k_0 \leq n+1$.

\item $n\equiv  1\ (\textrm{mod}\ 8)$. Assume that $E\in\{ E_{\Delta}\}_{\R}$. Since $\dim{\Delta}=2^n \geq n+1$, it follows that if $\text{rank}(E\oplus k) \geq 2^n$ then $E\oplus k$ is isomorphic to $E_{\Delta}\oplus k' =E_{\Delta\oplus k'}$, so $k_0 \leq 2^n$.

\item $n\equiv 2\ (\textrm{mod}\ 8)$ is analogue to the case $n\equiv  1\ (\textrm{mod}\ 8)$ since $\dim{r(\Delta_+)}= 2\cdot 2^{n-1}=2^n$.

\end{itemize}
\medskip
For the remaining cases we need the so-called Bott Integrability Theorem:

\begin{thm}[Corollary 9.8 in \cite{Hu}, Chapter 20]
Let $a\in H^{2n}(\mathbb{S}^{2n},\Z)$ be a generator. Then for each complex vector bundle $E$ over $\mathbb{S}^{2n}$, the $n$-th Chern class $c_n(E)$ is a multiple of $(n-1)!a$, and for each $m\equiv 0\ (\textrm{mod}\ (n - 1)!)$ there exists a unique $\{ E\}_{\C}\in S_{\C}(\mathbb{S}^{2n})$ such that $c_n (E)=ma$.
\end{thm}
Denote by $c_T(E)$ the total Chern class of $E$. Since $H^{*}(\mathbb{S}^{2n},\Z)=H^{0}(\mathbb{S}^{2n},\Z)\oplus H^{2n}(\mathbb{S}^{2n},\Z)$, it follows that $c_T(E)=1+ c_n (t)$, and from the Whitney sum Formula, $c_n (E\oplus F)=c_n (E)+c_n(F)$. It follows that 
\begin{equation}\label{equation chern class}
c_n (lE_{\Delta_{\pm}})=(n - 1)!(\pm l)a
\end{equation}

Now we return to the real setting, so let $E$ be again an arbitrary real vector bundle over the sphere $\mathbb{S}^n$.

\begin{itemize}

\item $n\equiv  0\ (\textrm{mod}\ 8)$. Assume that $E\in \pm l\{ E_{\Delta_+}\}_{\R}=\{ E_{l\Delta_\pm}\}_{\R}$, for some positive integer $l$. Since $\dim{l\Delta_\pm}=2^{n-1}l \geq n+1$, it follows that if $\text{rank}(E\oplus k) \geq 2^{n-1} l$, then $E\oplus k$ is isomoprhic to $E_{l\Delta_+}\oplus k' =E_{l\Delta_+\oplus k'}$, so $k_0 \leq 2^{n-1}l$.

The $(n/2)$-th Chern class of the complexified vector bundle $c(E)$ satisfies
$$c_{n/2}\left(c(E)\right)= c_{n/2}\left( c\left( E_{l\Delta_\pm}\right)\right)=c_{n/2}\left(E_{l\Delta_\pm}\right)= ((n/2) - 1)!(\pm l)a$$
where the first equality follows from the stability of the Chern classes and the last one from (\ref{equation chern class}).

\item $n\equiv 4\ (\textrm{mod}\ 8)$. Assume that $E\in \pm l\{ E_{r(\Delta_+)}\}_{\R}=\{ E_{lr(\Delta_\pm})\}_{\R}$, for some positive integer $l$. Since $\dim{lr(\Delta_\pm )}=2\cdot 2^{n-1}l \geq n+1$, it follows that if $\text{rank}(E\oplus k) \geq 2^{n} l$, then $E\oplus k$ is isomorphic to $E_{lr(\Delta_{\pm})}\oplus k' =E_{lr(\Delta_{\pm})\oplus k'}$, so $k_0 \leq 2^{n}l$.

The $(n/2)$-th Chern class of the complexified vector bundle $c(E)$ satisfies
$$c_{n/2}\left(c(E)\right)= c_{n/2}\left( c\left( E_{lr(\Delta_\pm)}\right)\right)= c_{n/2}\left( c\circ r\left( E_{l\Delta_\pm}\right)\right)$$
Now recall that $c\circ r =1+t$, where $t$ denotes the conjugation of complex vector bundles, so
$$c_{n/2}\left( cr\left( E_{l\Delta_\pm}\right)\right)=c_{n/2}\left(  E_{l\Delta_\pm}\oplus t\left( E_{l\Delta_\pm}\right)\right)=c_{n/2}\left(  E_{l\Delta_\pm}\right)+ c_{n/2}\left( t\left( E_{l\Delta_\pm}\right)\right)$$
The Chern class of the conjugate bundle satisfies (see Proposition 11.1 in \cite{Hu}, Chapter 17): 
$$c_{n/2}\left( t\left( E_{l\Delta_\pm}\right)\right)= (-1)^{n/2}c_{n/2}\left(  E_{l\Delta_\pm}\right)=c_{n/2}\left(  E_{l\Delta_\pm}\right)$$
since $n/2$ is even. So from (\ref{equation chern class}) we get that
$$c_{n/2}\left(c(E)\right)=2c_{n/2}\left(  E_{l\Delta_\pm}\right)=  ((n/2) - 1)!(\pm 2l)a, $$
which together with the inequality $k_0 \leq 2^n l$ above  proves the Theorem.
\end{itemize}

Finally, recall that the $k$-th Pontryagin class $p_k(E)\in H^{4k}(M,\Z)$ of a real vector bundle $E$ over a compact manifold $M$ is defined as:
$$p_k(E)=(-1)^k c_{2k}(c(E))$$
Therefore when $M$ is the sphere $\mathbb{S}^{n}$ of dimension $n\equiv 0\ (\textrm{mod}\ 8)$ (resp. $n\equiv 4\ (\textrm{mod}\ 8)$), we get that $p_{n/4} (E)=c_{n/2}(c(E))$ (resp. $p_{n/4} (E)=-1c_{n/2}(c(E))$). Anyway, in both cases
$$p_{n/4}(E)= ((n/2) - 1)!(\pm l_E) a$$
for some natural number $l_E$, where $a$ is a generator of $H^{n}(\mathbb{S}^{n},\Z)$.


\section{Grassmannian manifolds}\label{section grassmannian}

In this section $\F$ will stand for $\R$, $\C$ or $\Hq$. Let $\U_{\F}(n)$ denote the orthogonal group $\Or(n)$, the unitary group $\U(n)$ or the symplectic group $\Sp(n)\subset \U(2n)$ for $\F=\R$, $\C$ or $\Hq$ respectively. Throughout this section we will consider each of the groups $\U_{\F}(n)$ endowed with its canonical biinvariant metric.

The Grassmannian manifold $\Gr_{\F}(k,n)$ of $k$-dimensional subspaces of $\F^n$ (right subspaces in the case of $\Hq^{n}$) can be viewed as the homogeneous space $\U_{\F}(n)/ (\U_{\F}(k)\times \U_{\F}(n-k))$. This way, $\Gr_{\F}(k,n)$ inherits a quotient metric with nonnegative sectional curvature.

\subsection{Tautological bundle}\label{subsection tautological}
The \emph{tautological vector bundle} $\T_{\F}(k,n)$ over $\Gr_{\F}(k,n)$ is defined as
$$\T_{\F}(k,n)=\left\{ (W,w)\in \Gr_{\F}(k,n)\times\F^n : w\in W\right\}$$
where the bundle projection map is given by $(W,w)\mapsto W$. Define the representation:
$$\begin{array}{ccc}
 \rho_{\F}: \U_{\F}(k)\times \U_{\F}(n-k) & \longrightarrow & \U_{\F}(k)\\   (A,B) & \longmapsto & A 
 \end{array}
 $$
It turns out that $\T_{\F}(k,n)$ is isomorphic to the homogeneous vector bundle $E_{\rho_{\F}}$. The isomorphism is given by:
$$\begin{array}{ccc}
E_{\rho_{\F}} & \longrightarrow & \T_{\F}(k,n) \\ 
\left[M,v\right] & \longmapsto & \left(  M\left(
\begin{array}{c}
\F^k\\
0
\end{array}
\right) , M\left(\begin{array}{c}
v\\
0
\end{array}\right)\right) \end{array}$$
Notice that, although $\T_{\Hq}(k,n)$  is defined as a quaternionic vector bundle, here we are only considering its underlying complex structure. As such, it is isomorphic to the complex vector bundle 
$$E_{\rho_{\Hq}}=(\U_{\F}(n)\times \C^{2k})/ (\U_{\F}(k)\times \U_{\F}(n-k))$$

Observe that $\Gr_{\F}(k,n)$ is diffeomorphic to $\Gr_{\F}(n-k,n)$ under the map $W\mapsto W^{\perp}$, where $\F^n$ is endowed with the Euclidean metric. Clearly, the Whitney sum of $\T_{\F}(n-k,n)$ with $\T_{\F}(k,n)$ is the trivial bundle of rank $n$. From now on we will write just $\T_{\F}$ to denote the bundle $\T_{\F}(k,n)$.


\subsection{Proof of the Main Theorem for projective spaces}

Recall that $\Gr_{\F}(1,n+1)$ is the projective space $\FP^{n}$. In these cases, the quotient metric inherited from $\U_{\F}(n+1)$ is the one giving $\FP^{n}$ the structure of compact rank one symmetric space.

\begin{prop}\label{stablethm proj}
For $\F=\R$, $\C$ and $\Hq$, the following maps are surjective:
$$ \{\alpha\}_{\R}: \Rep_{\R}(\U_{\F}(1)\times \U_{\F}(n)) \rightarrow  S_{\R}(\FP^{n}) $$
$$ \{\alpha\}_{\C}: \Rep_{\C}(\U_{\F}(1)\times \U_{\F}(n)) \rightarrow  S_{\C}(\FP^{n}) $$
\end{prop}

\begin{proof}

The real and complex $K$-theory of projective spaces are well known, see for example \cite{Ad} and \cite{San}. The rings $\tilde{K}_{\R}(\FP^{n})$ and $\tilde{K}_{\C}(\FP^{n})$ are respectively generated by the following elements:
$$
\begin{array}{rcr}
\T_{\R}-1_{\R}\in\tilde{K}_{\R}(\RP^{n}),   &  \quad   &    c(\T_{\R})-1_{\C}\in\tilde{K}_{\C}(\RP^{n}),   \\
r(\T_{\C})-2_{\R}\in\tilde{K}_{\R}(\CP^{n}), &  \quad  &    \T_{\C}-1_{\C}\in\tilde{K}_{\C}(\CP^{n}),     \\
r(\T_{\Hq})-4_{\R}\in\tilde{K}_{\R}(\HP^{n}), & \quad  &    \T_{\Hq}-2_{\C}\in\tilde{K}_{\C}(\HP^{n}).
\end{array}
$$
Since the tautological bundle $\T_{\F}$ is homogeneous, the map
$$\alpha_{\R}: \tilde{\Rr}_{\R}(\U_{\F}(1)\times \U_{\F}(n))\rightarrow \tilde{K}_{\R}(\FP^{n})$$
is surjective, and by Lemma \ref{lem relates images beta},
$$ \{\alpha\}_{\R}: \Rep_{\R}(\U_{\F}(1)\times \U_{\F}(n)) \rightarrow  S_{\R}(\FP^{n})$$
is also surjective. The same arguments work for the map $\alpha_{\C}$.

\end{proof}

Proposition \ref{stablethm proj} proves that there is a homogeneous vector bundle in every stable class of real and complex vector bundles over each projective space. Now apply Lemma \ref{lemma nonnegative curvature} to get the Main Theorem for projective spaces.

\section{The Cayley plane}\label{section cayley plane}

In this section we consider the Cayley plane $\Ca$. Recall that the Cayley plane is a $16$-dimensional $CW$-complex consisting of three cells of dimensions $0$, $8$ and $16$. As a homogeneous space, it can be viewed as the quotient of the $52$-dimensional exceptional Lie group $F_4$ under the action of the spin group $\Spin(9)$. Let us endow $F_4$ with its canonical biinvariant metric, so that $\Ca$ with the quotient metric is a compact rank one symmetric space. 


\subsection{Representation rings $\Rr_{\F}(F_4)$ and $\Rr_{\F}(\Spin(9))$}

The representation rings of $F_4$ are known (see \cite{AdMM}, \cite{Mi2} and \cite{Yo}). Denote by $\lambda^k$ the $k$-th exterior product of the irreducible $26$-dimensional representation $\lambda$ given in Corollary 8.1 in \cite{AdMM}, and by $\kappa$ the adjoint action of $F_4$ on its Lie algebra $\mathfrak{f}_4$. It turns out that the representations $\lambda^k$ and $\kappa$ are real. We denote their complexifications in the same way. The real and complex representation rings of $F_4$ are the polynomial ring 
$$\Rr_{\F}(F_4)=\Z [\lambda^1,\lambda^{2},\lambda^{3},\kappa],$$
where $\F$ stands for $\R$ or $\C$, and the complexification map 
$$c: \Rr_{\R}(F_4)\rightarrow \Rr_{\C}(F_4)$$
is an isomorphism. 

The representation rings of $\Spin(9)$ have been described in Section \ref{subsection representation rings of spin}. Observe that the complexification map
$$c: \Rr_{\R}(\Spin(9))\rightarrow \Rr_{\C}(\Spin(9))$$
is surjective.







\medskip

\subsection{The $K_{\F}$-theory of $\Ca$}\label{subsection k-theory cayley}

The cohomology of $\Ca$ is well known, in particular we have:
$$
H^k(\Ca ,\Z)=\begin{cases} 
\Z  & \text{ if } k=0,8,16\\
0   &  \text{ otherwise }
\end{cases} 
$$
Hence $H^*(\Ca ,\Z)$ is torsion-free and Theorem \ref{thm atiyah hirze even dimen betti numbers} gives us the following:
$$K_{\C}(\Ca)=\Z\oplus \Z\oplus  \Z.$$ 
The real $K$-theory of $\Ca$ follows from Lemma 2.5  in \cite{Ho}, which states that if $M$ is a finite $CW$-complex with cells only in dimensions $ 0\ (\textrm{mod}\ 4)$ then 
$$K_{\R}(M)=\underbrace{\Z\oplus \dots\oplus\Z}_{n\ \text{times}}, $$
where $n$ is the number of cells in $M$. In particular, we have

\begin{prop}\label{prop real k-theory of cayley}
$K_{\R}(\Ca)=\Z\oplus \Z\oplus  \Z$.
\end{prop}

Now consider the induced map $r\circ c:K_{\R}(\Ca)\rightarrow K_{\R}(\Ca)$. By Proposition \ref{prop real k-theory of cayley} we know that $K_{\R}(\Ca)$ is torsion-free, and since the map $r\circ c$ is nothing but multiplication by $2$, it must be injective.

\begin{lem}\label{lem K-theory isomor Ca}
The induced map $r\circ c:K_{\R}(\Ca)\rightarrow K_{\R}(\Ca)$ is injective. In particular, $c:K_{\R}(\Ca)\rightarrow K_{\C}(\Ca)$ is also injective.
\end{lem}


\subsection{Proof of the Main Theorem for $\Ca$}

First we construct homogeneous bundles in every stable class.

\begin{prop}\label{propo stable cay}
The map
$$ \{\alpha\}_{\F}: \Rep_{\F}(\Spin(9)) \rightarrow  S_{\F}(\Ca) $$
is surjective for $\F=\R$ and $\C$. 
\end{prop}

\begin{proof}

The surjectivity of $\{\alpha\}_{\C}$ is included in Theorem \ref{thm complex bundles over many spaces} since $F_4$ is simply connected  and contains $\Spin(9)$ as a subgroup of maximal rank.


For the real case let $E$ be an arbitrary real vector bundle over $\Ca$. By the discussion above we have that
\begin{equation}\label{equation complex 1}
c(E-\text{rank}_{\R}{E} )=E_{\rho}-\dim{\rho}
\end{equation}
for some $\rho\in\Rep_{\C}(\Spin(9))$. On the other hand
$$c: \Rr_{\R}(\Spin(9))\rightarrow \Rr_{\C}(\Spin(9))$$
is surjective, so there exists $\rho'\in\Rep_{\R}(\Spin(9))$ such that $c(\rho')=\rho$, so
\begin{equation}\label{equation complex 2}
c(E_{\rho'}-\dim{\rho'})=E_{\rho}-\dim{\rho}
\end{equation}
By Lemma \ref{lem K-theory isomor Ca}, the complexification map $c:K_{\R}(\Ca)\rightarrow K_{\C}(\Ca)$ is injective, so from (\ref{equation complex 1}) and (\ref{equation complex 2}) it follows that
$$E_{\rho'}-\dim{\rho'}= E-\text{rank}_{\R}{E}$$
in $K_{\R}(\Ca)$ and hence $\{ E\}_{\R}= \{ E_{\rho'}\}_{\R}$.

\end{proof}

Finally, the proof of the Main Theorem for the Cayley plane is a direct consequence of Proposition  \ref{propo stable cay} together with Lemma \ref{lemma nonnegative curvature}.


\end{document}